\pdfoutput=1
\documentclass[letterpaper]{amsart}
\usepackage{amssymb}
\usepackage{enumerate}
\usepackage[final]{hyperref}
\usepackage{microtype}

\newtheorem{theorem}{Theorem}
\newtheorem{claim}{Claim}
\newtheorem{lemma}{Lemma}

\theoremstyle{definition}
\newtheorem{definition}{Definition}

\newcommand{\eps}		{\varepsilon}

\newcommand{\bA}		{\mathbb{A}}
\newcommand{\bB}		{\mathbb{B}}
\newcommand{\bC}		{\mathbb{C}}
\newcommand{\bD}		{\mathbb{D}}

\newcommand{\cA}		{\mathcal{A}}
\newcommand{\cB}		{\mathcal{B}}
\newcommand{\cC}		{\mathcal{C}}
\newcommand{\cS}		{\mathcal{S}}
\newcommand{\cF}		{\mathcal{F}}

\newcommand{\norm}[1]	{\lVert#1\rVert}

\renewcommand{\Pr}[2][]	{\mathbf{P}_{#1}\left( #2 \right)}

\newcommand{\E}		{\mathbf{E}}
\newcommand{\ind}[1]	{\mathbf{1}_{#1}}

\newcommand{\lam}		{\lambda_{\ast}}
\newcommand{\K}		{\kappa}

\DeclareMathOperator{\essinf}{ess\,inf}
\DeclareMathOperator{\esssup}{ess\,sup} 
\DeclareMathOperator{\diam}{diam}

\begin{document}

\title{The diameter of Inhomogeneous Random Graphs}

\author{Nicolas Fraiman}
\address{Department of Mathematics\\ Harvard University\\ One Oxford Street, Cambridge MA 02138, USA}
\email{fraiman@math.harvard.edu}

\author{Dieter Mitsche}
\address{Laboratoire J--A Dieudonn\'e\\ Universit\'e Nice Sophia Antipolis\\ Parc Valrose, 06108 Nice Cedex 02, France}
\email{dmitsche@unice.fr}

\thanks{N.F. was partially supported by an NSERC Postdoctoral Fellowship and the Department of Mathematics at University of Pennsylvania.} 

\date{\today}

\maketitle

\begin{abstract}
In this paper we study the diameter of inhomogeneous random graphs $G(n,\K,p)$ that are induced by irreducible kernels $\K$. The kernels we consider act on separable metric spaces and are almost everywhere continuous. We generalize results known for the Erd\H{o}s-R\'enyi model $G(n,p)$ for several ranges of $p$.  We find upper and lower bounds for the diameter of $G(n,\K,p)$ in terms of the expansion factor and two explicit constants that depend on the behavior of the kernel over partitions of the metric space.
\end{abstract}

\section{Introduction}\label{sec:intro}

In this work we study metric properties of inhomogeneous random graphs, where edges are present independently but with unequal edge occupation probabilities. We study the behavior of the diameter for different ranges of the mean edge density. Under weak assumptions we find tight asymptotic bounds of the diameter for connected graphs in this random graph model.

Let $\cS$ be a separable metric space and $\mu$ a Borel probability measure on $\cS$. Let $\K:\cS\times\cS \to [0,1]$ be a measurable symmetric kernel. The \emph{inhomogeneous random graph} with kernel $\K$ and density parameter $p$ (depending on $n$) is the random graph $G(n,\K, p) = (V,E)$ where the vertex set is $V = \{1,\dots, n\}$ and we connect each pair of vertices $i,j\in V$ independently with probability $p_{ij} = \K(X_i,X_j)p$, where  $X_1,\dots,X_n$ are independent $\mu$-distributed random variables on $\cS$. 

We study the asymptotic expansions for distances in the graph $G(n,\K, p)$ by associating to $G(n,\K, p)$ two graphs induced by the kernel $\K$. Given two subsets $\cA,\cB\subset \cS$, let 
\begin{align*}
\K_\ell(\cA,\cB) &= \essinf \{\K(x,y): x\in \cA, y\in \cB\}, \\
\K_u(\cA,\cB) &= \esssup \{\K(x,y): x\in \cA, y\in \cB\}.
\end{align*}
For a partition $\bA=\{\cA_1,\dots,\cA_m\}$ of $\cS$, we define the \emph{lower partition graph} $P_\ell(\bA)$ induced by $\bA$ as the graph with vertex set $\bA$ and where $(\cA_i,\cA_j)$ is an edge if $\K_\ell(\cA_i, \cA_j)> 0$. Analogously, we define the \emph{upper partition graph} $P_u(\bA)$  as the graph with vertex set $\bA$ and where $(\cA_i,\cA_j)$ is an edge if  $\K_u(\cA_i, \cA_j)> 0$. 

A kernel $\K$ on $(\cS,\mu)$ is \emph{reducible} if there exists a set $\cA \subset \cS$ with $0 < \mu(\cA) < 1$ such that $\K = 0$ almost everywhere on $\cA \times (\cS\setminus\cA)$. Otherwise $\K$ is \emph{irreducible}.

Throughout the paper we always assume that $n$ is sufficiently large. We say that a sequence of events holds \emph{with high probability}, if it holds with probability tending to $1$ as $n\to\infty$. Since we are only interested in results that hold with high probability, we avoid working with partitions that have measure zero sets. Therefore,  every time we refer to a partition we assume that is has no measure zero sets.

For two vertices $u,v \in V$ belonging to the same connected component of a graph $G=(V,E)$, denote by $d_G(u,v)$ the graph distance between $u$ and $v$, that is, the number of edges on a shortest path between them. For a connected graph $G$, let $\diam G=\max_{u,v} d_G(u,v)$. 
We study the diameter of $G(n,\K, p)$ by studying the diameters of the induced graphs $P_\ell(\bA)$  and $P_u(\bA)$. 

Our main result is then stated in terms of the following two constants
\[
\Delta_\ell := \inf_\bA \diam P_\ell(\bA)
\qquad\text{and}\qquad
\Delta_u := \sup_\bA \diam P_u(\bA),
\]
where $\bA$ ranges over all partitions with no measure zero sets.  In Section \ref{sec:partitions} we show our first result:
\begin{theorem}\label{thm:diff2}
Suppose $\Delta_\ell < \infty$. Then
\[
\Delta_u \leq \Delta_\ell \leq \Delta_u +2,
\]
both bounds being tight.
\end{theorem}
In order to introduce our next theorem we need to  define the \emph{expansion factor} 
\[
\Phi := \left\lceil \frac{\log n }{ \log np} \right\rceil.
\] 
This quantity is about the diameter of $G(n,p)$, where $G(n,p)$ is the Erd\H{o}s-R\'enyi graph, as first shown  by \cite{Bol81}. In order to simplify our main result (see below for more details), we will assume that $\log n/\log np$ is at least $2\log \log n/\log np$ away from every integer. 

Our main result is the following:
\begin{theorem}\label{thm:main} 
Let $\K$ be an irreducible kernel that is continuous $(\mu\otimes\mu)$-almost everywhere and let $G(n,\K, p)$ be the induced inhomogeneous random graph. With high probability, the following statements hold:
\begin{enumerate}[(i)]
\item If $\Phi <\Delta_u$, then \[ \Delta_u \leq \diam G(n,\K,p) \leq \Delta_\ell. \] 

\item If $\Delta_u \leq  \Phi < \Delta_\ell$, then  \[ \Phi \leq \diam G(n,\K,p) \leq \Delta_\ell. \] 
Moreover,  if for every  partition $\bA$ there exist $\cA_i$ and $\cA_j$, with no walk of length exactly $\Phi$ between them in $P_u(\bA)$, then \[ \Phi+1 \leq \diam G(n,\K,p). \] 
\item If $ \Delta_\ell \leq \Phi$, then  \[ \Phi  \leq \diam G(n,\K,p) \leq \Phi+1. \] 
	Moreover, \[ \diam G(n,\K,p) = \Phi +1, \] iff for every  partition $\bA$ there exist $\cA_i$ and $\cA_j$, with no walk of length exactly $\Phi$ between them in $P_u(\bA)$. 
\end{enumerate}
\end{theorem}

\subsection{Background and history}
A discrete version of this model was introduced by S{\"o}derberg \cite{Sod02}. The sparse case (when the number of edges is linear in the number $n$ of vertices) was studied in detail in by Bollob\'as, Janson and Riordan \cite{BJR07}. Among other things they give an asymptotic formula for the diameter of the giant component when it exists. Connectivity at the intermediate case was analyzed by Devroye and Fraiman \cite{DF14}. The dense case (when the number of edges is quadratic in $n$) is closely related with the theory of graph limits started by Lov\'asz and Szegedy \cite{LS06} and further studied in depth by Borgs, Chayes, Lov\'asz, S\'os and Vesztergombi \cite{BCLSV08,BCLSV12} among others. For a thorough introduction to the subject of graph limits see the book by Lov\'asz \cite{Lov12}.

The diameter of random graphs has been studied widely. In particular, for the Erd\H{o}s-R\'enyi model, Bollob\'as \cite{Bol81} generalized the results from Klee and Larman \cite{KL81} characterizing the case of constant diameter. Later, Chung and Lu \cite{CL01} proved concentration results in various different ranges. More recently, Riordan and Wormald \cite{RW10} completed the program to study the missing cases for the Erd\H{o}s-R\'enyi model.

The critical window, when $p = 1/n+c n^{-4/3}$, for $G(n,p)$ is much harder to analyze. Nachmias and Peres \cite{NP08} obtained the order of the diameter, namely $n^{1/3}$. Addario--Berry, Broutin and Goldschmidt \cite{ABG10,ABG12} proved convergence, in the Gromov--Hausdorff distance, of the rescaled connected components to a sequence of continuous compact metric spaces. In particular, the diameter rescaled by $n^{-1/3}$ converges in distribution to an absolutely continuous random variable with finite mean. Their approach was extended by Bhamidi, Sen and Wang \cite{BSW14} to the Norros--Reittu \cite{NR06} random graph model, and then further generalized by Bhamidi, Broutin, Sen and Wang \cite{BBSW14}.

\subsection{Structure of the paper} 
In Section \ref{sec:frame} we introduce all concepts, additional definitions and results needed to prove Theorem \ref{thm:main}. In Section \ref{sec:partitions} we prove Theorem~\ref{thm:diff2} on the behavior of the upper and lower diameters $\Delta_u$ and $\Delta_\ell$. In Section  \ref{sec:expansion} we prove that the number  of vertices that are at a fixed distance from a given vertex grows exponentially as a function of the distance. Finally, in Section \ref{sec:diameter} we combine the results of the previous sections to give the proof of Theorem \ref{thm:main}.

\section{Framework}\label{sec:frame}

In this paper we follow the notation from \cite{BJR07} with minor changes. We also use the following standard notation: we write $f =O(g)$ if $f/g$ is bounded and $f = o(g)$ if $f/g\to 0$.

Define $N(u) = \{v\in V: (u,v)\in E\}$ the neighborhood of vertex $u$, and for a subset $U\subset V$ let $N(U) = \cup_{u\in U} N(u)$. 
 Given a subset $\cA \subset \cS$ we write $V(\cA)$ for the set of vertices with type in $\cA$, i.e.,
\[
V(\cA) = \left\{v\in V: X_v \in \cA \big.\right\}.
\]

The asymptotic expansions for distances in the graph $G(n,\K,p)$ are obtained by looking at the lower and upper partition graphs $P_\ell(\bA)$  and $P_u(\bA)$ of a partition $\bA=\{\cA_1,\dots,\cA_m\}$ of $\cS$, as defined in the introduction.  These graphs are finite graphs that describe approximations of $\K$ that may be successively refined. More formally, we have the following definition:

\begin{definition}\label{def:refinement}
We say that a partition $\bA$ is a \emph{refinement} of $\bB$, denoted by $\bA \prec \bB$, if for every $\cA\in\bA$ there exists $\cB\in\bB$ such that $\cA\subset\cB$. Note that in this case, each $\cB_i = \cup_{p=1}^{m_i} \cA_p^{(i)}$ $\mu$-almost everywhere.
\end{definition}
 
Let us examine the effect of a refinement on the partition graphs:  
It is clear that $\K_u(\cB_i,\cB_j) > 0$ if and only if there exists $p,q$ and $\cA_i, \cA_j$ with $\cA_i\subset \cB_i$ and $\cA_j \subset \cB_j$ such that $\K_u(\cA_i,\cA_j) > 0$. This implies that $P_u(\bB)$ is obtained from $P_u(\bA)$ by contracting the vertices $\cA_i \subset \cB_i$ into one vertex $\cB_i$. In particular, 
\[
\diam P_u(\bB) \leq \diam P_u(\bA),
\] 
On the other hand, $\K_\ell(\cB_i,\cB_j) > 0$ if and only if for all $\cA_i\subset \cB_i$ and $\cA_j \subset\cB_j$ we have $\K_\ell(\cA_i,\cA_j) > 0$. This implies that the graph obtained by splitting each vertex $\cB_i$ into the parts of $\bA$ that it contains is a subgraph of $P_\ell(\bA)$. In particular, 
\[
\diam P_\ell(\bB) \geq \diam P_\ell(\bA).
\]
Note also that if $\bA \prec \bB$ and $\cA_i \subset \cB_i$ and $\cA_j \subset \cB_j$, we always have 
\begin{align*}
d_{P_\ell(\bA)} (\cA_i,\cA_j) \leq d_{P_\ell(\bB)} (\cB_i,\cB_j), \\
d_{P_u(\bA)} (\cA_i,\cA_j) \geq d_{P_u(\bB)} (\cB_i,\cB_j).
\end{align*}
By the above properties it is clear that the values of $\Delta_\ell = \inf_\bA \diam P_\ell(\bA)$ and $\Delta_u = \sup_\bA \diam P_u(\bA)$ are well defined.
When studying $\Delta_\ell$ and $\Delta_u$  we want to avoid trivial cases where they are infinite because there is an structural obstruction for connectivity given by $\K$. If $\K$ is reducible then the whole graph $G(n,\K,p)$ is disconnected since almost surely there are no edges between the sets $V(\cA)$ and $V(\cS\setminus\cA)$. Since we want to work with connected graphs, we restrict our attention to the irreducible case. Let
\[
\lambda(x) = \int_\cS \K(x,y) d\mu(y) \qquad\text{and}\qquad \lam = \essinf \lambda(x).
\]
In \cite{DF14} the constant $\lam$ is called the \emph{isolation parameter} and it is shown to determine the connectivity threshold. Note that the expected degree of a vertex of type $x$ is $\lambda(x)np$. Our arguments require this value to be $\Theta(np)$ so we will work with kernels that satisfy $\lam > 0$. Observe that the assumption $\lam > 0$ implies that all lower partition graphs $P_\ell(\bA)$ are connected (the upper partition graphs are connected anyway, since $\K$ is assumed to be irreducible). 

Recall that we defined the expansion factor as $\Phi = \lceil \log n/\log np \rceil.$ Its relation to the diameter of the Erd\H{o}s and R\'{e}nyi model $G(n,p)$ is made precise now. Since we use it in our argument we include the following result obtained first by~\cite{Bol81} as a Lemma.

\begin{lemma}[\cite{Bol01}, Corollary 10.12]\label{lem:diameter} 
Let $k \geq 2$ be finite. If $(np)^k/n -2 \log n \to \infty$ and $(np)^{k-1}/n - 2\log n \to -\infty$. Then, with high probability $\diam(G(n,p))=k$.
\end{lemma}
Since we assume that $\log n/\log np$ is at least $2\log \log n/\log np$ away from every integer, clearly, the diameter of $G(n,p)$ is with high probability concentrated on one value, and hence clearly $k=\Phi$. In fact, this assumption is stronger and implies that
\begin{equation}\label{eq:diambounds}
\frac{(np)^{\Phi}}{n} -\omega \log n \to \infty 
\qquad\mbox{and}\qquad
\frac{(np)^{\Phi -1}}{n} - \frac{1}{\omega}\log n \to -\infty.
\end{equation}
The results hold also for the ranges in between with the obvious changes. Since this is handled as in $G(n,p)$, we focus on the one--value case for the sake of clarity.

\section{Upper and lower diameters}\label{sec:partitions} 

In this section we study the behavior of the diameters $\Delta_u$ and $\Delta_\ell$. The goal is to prove Theorem~\ref{thm:diff2}. We split the proof into two lemmas. Given two partitions $\bA$ and $\bB$, we define their \emph{common refinement} as
\[
\bA \vee \bB := \{\cA\cap\cB: \mu(\cA\cap\cB) > 0, \text{ for } \cA\in\bA \text{ and } \cB\in\bB \}.
\]

\begin{lemma}\label{lem:usmallerl}
If $\Delta_\ell < \infty$, then $\Delta_u \leq \Delta_\ell$.
\end{lemma}
\begin{proof}
Since $\Delta_\ell < \infty$, there exists a partition $\bA$ such that $\diam P_\ell(\bA) = \Delta_\ell$.
Let $\bB$ be an arbitrary partition. Consider the common refinement $\bA \vee \bB$. Since $P_\ell(\bA\vee \bB)$ is a subgraph of $P_u(\bA\vee \bB)$ we have that 
\[
\diam P_u(\bA\vee \bB) \leq \diam P_\ell(\bA\vee \bB).
\] 
Moreover, we also have that 
\begin{align*}
\diam P_u(\bB) \leq \diam P_u(\bA\vee \bB) \qquad \text{and} \qquad 
\diam P_\ell(\bA\vee \bB) \leq \diam P_\ell(\bA),
\end{align*}
because $\bA\vee \bB$ is a refinement of both $\bA$ and $\bB$. Combining these three inequalities we get that 
$\diam P_u(\bB) \leq \diam P_u(\bA) \le \diam P_\ell(\bA) = \Delta_\ell < \infty$. Therefore, the inequality also holds taking the supremum over all partitions. In particular, we can choose $\bB$ such that $\diam P_u(\bB) = \Delta_u$, and the desired inequality follows. 
\end{proof}

In particular, the above bound gives an easy way to determine $\Delta_u$ and $\Delta_\ell$ in the case they are equal. It suffices to find a partition $\bA$ that such that 
\[
\diam P_u(\bA) = \diam P_\ell(\bA).
\] 
We can also show the following bound.

\begin{lemma}\label{lem:lsmalleru}
If $\Delta_\ell < \infty$, then $\Delta_\ell \leq \Delta_u + 2$.
\end{lemma}
\begin{proof}
We state the following claim which we prove below.
\begin{claim}\label{claim1}
Suppose $\Delta_u < \infty$. Given a partition $\bB$, let $\cB_s,\cB_f \in \bB$. There exists a refinement $\bA \prec \bB$ such that there exist $\cA_s, \cA_f \in \bA$ with $c(\cA_s)=\cB_s$, $c(\cA_f)=\cB_f$ such that $d_{P_\ell(\bA)} (\cA_s,\cA_f) \leq \Delta_u$.
\end{claim}
Assuming the claim, the lemma follows easily: indeed, start with a partition $\bB$ with $\diam P_\ell(\bB) <\infty$. Consider all pairs 
\[
\mathfrak{P} = \{ (\cB_i,\cB_j) \in \bB\times\bB: d_{P_\ell(\bB)} (\cB_i,\cB_j) > \Delta_u + 2\}
\]
If $\mathfrak{P} = \emptyset$ then we are done, as $\Delta_\ell \leq \max_{\cB_i,\cB_j \in \bB} d_{P_\ell(\bB)} (\cB_i,\cB_j) \leq \Delta_u+2$. Therefore suppose $\mathfrak{P}  \neq \emptyset$. Since $\diam P_\ell(\bB) < \infty$ given $(\cB_i,\cB_j) \in \mathfrak{P}$ there exists $\cB_s$ and $\cB_f$ such that $\K_\ell(\cB_i,\cB_s) > 0$ and $\K_\ell(\cB_f,\cB_j) > 0$. 

Since $\Delta_\ell < \infty$,  by Lemma~\ref{lem:usmallerl} we have $\Delta_u \leq \Delta_\ell < \infty$. Thus, by Claim~\ref{claim1},  there exists $\bA \prec \bB$ such that $d_{P_\ell(\bA)} (\cA_s,\cA_f) \leq \Delta_u$ for some $\cA_s,\cA_f$ with $c(\cA_s)=\cB_s$ and $c(\cA_f)=\cB_f$. Then, for any $\cA_i,\cA_j\in \bA$ such that $c(\cA_i) = \cB_i$ and $c(\cA_j)=\cB_j$ we have that $\K_\ell(\cA_i,\cA_s) > 0$ and $\K_\ell(\cA_f,\cA_j) > 0$ and therefore $d_{P_\ell(\bA)}(\cA_i,\cA_j) \leq \Delta_u + 2$.

We construct such a partition $\bA$ for each pair in $\mathfrak{P}$. Since $\mathfrak{P}$ is finite consider a common refinement $\bC$ of all of these partitions. It is clear that $\bC$ has $\max_{\cC_i,\cC_j \in \bC} d_{P_\ell(\bC)} (\cC_i,\cC_j) \leq \Delta_u+2$, and since $\Delta_{\ell} \leq \max_{\cC_i,\cC_j \in \bC} d_{P_\ell(\bC)} (\cC_i,\cC_j)$, the lemma follows. \qedhere

\textbf{Proof of Claim~\ref{claim1}.}
Eliminate from $\cS$ all points where $\K$ is not continuous. Note that the removed set has measure $0$ and does not affect the calculations of essential infima and suprema. 

Next, suppose that in the remaining set there exist $x_1,\ldots, x_r$ with $r \leq \Delta_u$ with $x_1 \in \cB_1 := \cB_s, x_r \in \cB_r := \cB_f$, $x_i \in \cB_i$ for $i=2,\ldots,r-1$ and $\K(x_i,x_{i+1}) > 0$ for $i=1,\ldots,r-1$. By continuity of $\K$, there exist $\eps_i, \eps_{i+1} > 0$ such that for any $y$ in the ball $B(x_i,\eps_i)$ and $z$ in  the ball $B(x_{i+1},\eps_{i+1})$ we have $\K(y,z) > 0$. Consider the partition $\bA \prec \bB$ in the following way: all parts except for $\cB_i$ with $i=1,\ldots,r$ remain unchanged: for $i=1,\ldots,r$, $\cB_i$ is split into $\cA_i=\cB_i\cap B(x_i,\eps_i)$ and $\cA'_i=\cB_i \setminus B(x_i, \eps_i)$. Since in $\bA$ we have $d_{P_\ell(\bA)}(\cA_s,\cA_f)\leq r \leq \Delta_u$, we found the desired partition $\bA$. 

Otherwise, there exists no such path of length $r \leq \Delta_u$. Consider any path of minimal distance $\cB_1 := \cB_s, \cB_2,\ldots, \cB_r := \cB_f$ of length $r \leq \Delta_u$ in $P_u(\bB)$. For $i=1,\ldots,r$ let 
\[ 
\cA_i^s :=\{ x \in \cB_i: \exists (x_1,\ldots,x_i=x) \in (\cB_1,\ldots,\cB_i) \mid \K(x_j,x_{j+1}) > 0, j=1,\ldots,i-1 \}
\]
be the sets of vertices of $\cB_i$ to which there is a path that starts at $\cB_s$.
Similarly, 
\[ 
\cA_i^f :=\{ x \in \cB_i: \exists (x_i=x,\ldots,x_r) \in (\cB_i,\ldots,\cB_r) \mid \K(x_j,x_{j+1}) > 0, j=i,\ldots,r-1 \}
\]
be the sets of vertices of $\cB_i$ from which there is a path that finishes at $\cB_f$. 
Note that for all $i=1,\ldots,r$, we must have $\cA_i^s \cap \cA_i^f = \emptyset$, as otherwise we would have a path of length $r$ between $\cB_s$ and $\cB_f$. Consider the partition $\bA \prec \bB$ induced from splitting $\cB_i$ into $\cA_i^s, \cA_i^f$ and $\cB_i \setminus (\cA_i^s \cup \cA_i^f)$. Since some of these sets might be empty we consider the partition obtained after removing sets of measure zero. Note that for the new partition $\bA$, the shortest path starting from $\cA_1^s$ and ending at $\cA_r^f$ and using only elements $\cA_i^s, \cA_i^f, \cB_i \setminus (\cA_i^s \cup \cA_i^f)$ for some $i=1,\ldots,r$ must have length $d_1 \geq r+1$ in the upper partition graph corresponding to $\bA_1:=\bA$. If there are several minimal paths of length $r$ in $\bB$ between $\cB_s$ and $\cB_f$, do independently the same refinement and obtain for each such path a refined partition $\bA_i \prec \bB$. Note that there are only finitely many partitions, since there are only finitely many minimal paths of length $r$ in $\bB$. 
As before, when refining, distances in the upper partition graph either stay the same or increase. We may thus take a partition $\bC$ which is a common refinement of all $\bA_i$, and we have for all  $\cC_s, \cC_f \in \bC$ with $c(\cC_s)=\cB_s$ and $c(\cC_f)=\cB_f$, $d_{P_u(\bC)}(\cC_s,\cC_f) \geq d_1$. If $d_1 > \Delta_u$, we found a new partition $\bC$ with diameter bigger than $\Delta_u$, contradicting our assumption on $\Delta_u$. 
Otherwise we apply the claim with partition $\bC$ playing the role of partition $\bB$. Note that there are only finitely many elements $\cC_s, \cC_f$ with $c(\cC_s)=\cB_s$ and $c(\cC_f)=\cB_f$. For a fixed pair of such elements $\cC_s,\cC_f$ apply the claim (yielding a sequence of refined partitions corresponding  to all minimal paths of length $d$ between them), giving either the desired path or a partition being a common refinement of all these partitions. Taking then again the refinement of all refined partitions corresponding to all pairs $\cC_s, \cC_f$ yields a new refinement $\bD$ in which all pairs are at distance $d_2 \geq d_1+1$, and the claim can then be applied with $\bD$ playing the role of $\bC$. Since $\Delta_u < \infty$, after finitely many iterations we must have found the desired path of length at most $\Delta_u$, and the claim follows.\hfill\qed
\end{proof}

The first part of Theorem~\ref{thm:diff2} follows now easily by combining Lemma~\ref{lem:usmallerl} and Lemma~\ref{lem:lsmalleru}. For the second part, to show that both bounds can be attained, on the one hand consider a constant kernel defined as $\K(x,y)=1$ for all $x \leq  y$, and extended by symmetry for $x > y$. Clearly, for any partition $\bA$ of $\cS$, the upper and lower partitions corresponding to $\bA$ are the same graphs, and therefore for such a kernel $\Delta_{\ell}=\Delta_u$. On the other hand, to show that $\Delta_{\ell} = \Delta_u + 2$ can be attained, consider for example the following kernel: let $\cS=[0,k+2]$, and let  $\varepsilon > 0$ be a sufficiently small constant that ensures that consecutive intervals overlap below (note that by dividing all values by $k+2$ clearly the same result can be obtained when considering $\cS=[0,1]$). Then define $\K(x,y)$ for $x < y$ as follows:
\[
\K(x,y) = \ind{A} + \sum_{i=1}^k \ind{B_i\setminus A} \left(x- \frac{(i-1)(k+2)(1-\varepsilon)}{k}\right) \left(\frac{i(k+2)}{k} - y \right),
\]
where
\begin{align*}
A &= \Big\{(x,y)\in\cS\times\cS: y-x \leq 1 \Big\} \\
\intertext{and}
B_i &= \left\{(x,y)\in\cS\times\cS: \frac{(i-1)(k+2)(1-\varepsilon)}{k} \leq x < y \leq \frac{i (k+2)}{k} \right\}.
\end{align*}

Clearly, for any partition $\bA$ of $\cS$, for both the upper and lower partition graphs of $\bA$ the diameter of the corresponding graph is attained for the distance between $0$ and $k+2$. To see that $\Delta_u=k$ consider an arbitrarily fine partition $\bA$: the element of $\bA$ containing $0$ is connected by an edge in the upper partition graph to the element of $\bA$ containing $(k+2)/k$, this one is then connected by an edge to the element of $\bA$ containing $2(k+2)/k$, and continuing like this, the element containing $(k-1)(k+2)/k < k$ is then connected by an edge to the element of $\bA$ containing $k$, and we have found a path of length $k$. Clearly, in an arbitrarily fine partition these elements of all $\bA$ are all different, and no shorter path can be found, since the element of $\cA$ containing $i(k+2)/k$ is not connected by an edge to the element of $\cA$ containing only elements bigger than $(i+1)(k+2)/k$. By definition of $\K$ is certainly an optimal strategy, and  hence $\Delta_u \leq k$. To show that $\Delta_{\ell} \geq k+2$, observe that the element of $\bA$ containing $0$ can possibly be connected by an edge to the element of $\bA$ containing $1$, but not to an element of $\bA$ containing only elements of $\cS$ bigger than $1$. For $\varepsilon$ small enough, $(k+2)(1-\varepsilon)/k > 1$, and therefore in the next step the element of $\bA$ containing $1$ can possibly be connected by an edge at most to the element of $\bA$ containing $2$, but not to an element containing only elements of $\bA$ bigger than $2$. Repeating the same argument, for any $i=2,\ldots,k-1$, the element containing $i$ can possibly be connected by an edge at most to the element of $\bA$ containing $i+1$, but not to an element of $\bA$ containing only elements of $\cS$ bigger than $i+1$. By construction of $\K$ this is clearly an optimal strategy, and hence $\Delta_{\ell} \geq k+2$. By the first part of the theorem, we then must have $\Delta_u=k$ and $\Delta_{\ell}=k+2$. The proof of Theorem~\ref{thm:diff2} is complete.

\section{Expansion of neighborhoods}\label{sec:expansion}
Let $G(n,\K, p)$ be an inhomogeneous random graph induced by an irreducible kernel $\K$  that is continuous $(\mu\otimes\mu)$-almost everywhere.  
In this section we prove that the number of vertices at distance $k$ from a given vertex grows exponentially with $k$. The main tool we use is concentration inequalities. 

\begin{lemma}\label{lem:step}
Let $\bA=\{\cA_1,\dots,\cA_m\}$ be a partition of $\cS$. Let $U \subset V(\cA_i) \leq |V(\cA_i)|/2$. Let $\cA_j \in \bA$, with possibly $\cA_i=\cA_j$ and suppose that $\K_\ell(\cA_i,\cA_j) > 0$. Define 
\[S_{i,j} := n\mu(\cA_j) \left( 1-(1-\K_\ell(\cA_i,\cA_j)p)^{|U|}\right) 
= \begin{cases} 
n\mu(\cA_j)(1+o(1)), & p|U|=\omega(1), \smallskip \\ 
\Theta(n \mu(\cA_j)), & p|U|=\Theta(1), \smallskip \\ 
n\mu(\cA_j)p|U|(1+o(1)), & p|U|=o(1).
\end{cases}
\]
Then
\[
\Pr{ | N(U)\cap (V(\cA_j) \setminus U) | \leq \frac{ S_{i,j}}{4}  \big.} \leq e^{-S_{i,j}/16}.
\]
\end{lemma}

\begin{proof}
We write the number of vertices as a sum of indicators
\[
| N(U)\cap (V(\cA_j) \setminus U) | = \sum_{v\notin U} \ind{X_v\in \cA_j} \ind{v \in N(U)}.
\]
Note that, for different values of $v\notin U$, the events $[X_v\in \cA_j]\cap [v \in N(U)]$ are independent and identically distributed. Fix a vertex $v \notin U$. We want to bound the probability of these events conditionally on $[U\subset V(\cA_i)]$. We begin by noting that 
\begin{equation}\label{eq:setprob}
\Pr{X_v\in \cA_j \mid U\subset V(\cA_i)\big.} = \Pr{X_v\in \cA_j} = \mu(\cA_j),
\end{equation}
Moreover, by independence we have
\begin{align*}
\Pr{v \notin N(U) \mid U\subset V(\cA_i), X_v\in \cA_j \big.} 
&= \prod_{u\in U} \Pr{v \notin N(u) \mid X_u\in \cA_i, X_v\in \cA_j \big.} \\
&\leq (1-\K_\ell(\cA_i,\cA_j)p )^{|U|}.
\end{align*}
Therefore, we can bound 
\begin{align}\label{eq:linkprob}
\Pr{v \in N(U) \mid U\subset V(\cA_i), X_v\in \cA_j \big.}  &\geq 1-(1-\K_\ell(\cA_i,\cA_j)p)^{|U|} 
\end{align}
Putting equations \eqref{eq:setprob} and \eqref{eq:linkprob} together we obtain
\begin{align*}
&\Pr{X_v\in \cA_j, v \in N(U) \mid U\subset V(\cA_i)\big.}  \\
&= \Pr{X_v\in \cA_j \mid U\subset V(\cA_i)\big.} \cdot \Pr{v \in N(U) \mid U\subset V(\cA_i), X_v\in \cA_j \big.} \\
&\geq  \mu(\cA_j) \left(1-(1-\K_\ell(\cA_i,\cA_j)p)^{|U|} \right)
\end{align*}
The proof is complete by applying a Chernoff bound \cite{Che52} to $| N(U)\cap (V(\cA_j)\setminus U) |$ noting that $S_{i,j}/2 \leq \E\; | N(U)\cap (V(\cA_j) \setminus U) |$.
\end{proof}
We will the previous lemma iteratively to get concentrations for expansions of the $i$-th neighborhood. 
We need a few definitions: let
\[ 
\K_\ell := \min \{ \K_\ell(\cA_i,\cA_j):  (\cA_i,\cA_j) \text{ is an edge of $P_\ell(\bA)$} \} 
\]
and for a partition $\bA=\{\cA_1,\ldots, \cA_r\}$ set
\[
\norm{\bA}_\mu = \min \{ \mu(\cA_i): \cA_i \in \bA \}.
\]
Let $S := n\norm{\bA}_{\mu} \K_\ell p/4$ and note that $S_{i,j}/4 \geq S$ for all $\cA_j, \cA_i \in \bA$.  Define $t(k)=k$ for $0 \leq k \leq \Phi-2$ and $t(k)=\Phi-2$ for $k > \Phi-2$.
Given a partition $\bA$, consider a walk $\cA_0,\cA_1,\dots,\cA_\ell$ in $P_\ell(\bA)$ (possibly making zigzags and reusing partitions) and let $u\in V(\cA_0)$. Define $\Gamma_0(u) = \{ u \}$, and for $k \geq 1$, define first recursively the following auxiliary variables $\Gamma'_k(u)$: for $|\Gamma_{k-1}(u)| \leq |V(\cA_{k-1})|/2$, let $\Gamma'_{k-1}(u):=\Gamma_{k-1}(u)$, and otherwise define $\Gamma'_{k-1}(u) \subseteq \Gamma_{k-1}(u)$ to be a randomly chosen subset of $\Gamma_{k-1}(u)$ such that $|\Gamma'_{k-1}(u)|=\lfloor |V(\cA_{k-1})|/2 \rfloor$. Define then
\[
\Gamma_k(u) = \left\{ v\in V(\cA_k)\setminus \Gamma'_{k-1}(u): N(v)\cap \Gamma'_{k-1}(u)\neq \emptyset \Big.\right\}.
\]

We can now prove the following lemma.

\begin{lemma}\label{lem:path}
Let $\omega$ be a function tending to infinity arbitrarily slowly as $n \to \infty$, let $p \geq \omega \log n/n$ and let $\ell \leq \omega$. Let $\bA$ be a partition of $\K$ and let $\cA_0,\cA_1,\dots,\cA_\ell$ be a walk in $P_\ell(\bA)$, and let $u\in V(\cA_0)$. Then, with probability $1-o(n^{-\omega})$,
\[
 |\Gamma_\ell(u)| \geq S^{t(\ell)}.
\]
 \end{lemma}
\begin{proof}
Note that $\Gamma_k(u) = N(\Gamma'_{k-1}(u))\cap ( V(\cA_k) \setminus \Gamma'_{k-1}(u))$. Define the events $\cF_0 = \Omega$ and in general by induction, for $k \geq 1$,  
\[
\cF_k = \bigcap_{i=0}^{k-1} \cF_i \cap \left[ | \Gamma_k(u) | \geq S^{t(k)}\right].
\]
Note that if $\cF_k$ holds, we have that $|\Gamma_i(u)| \geq S^{t(i)}$ for all $i=0,\dots,k$. We will now bound $\Pr{\cF_k^c | \cF_{k-1}}$ from above and condition under $\cF_{k-1}$. If  $| \Gamma'_{k-1}(u)|p=o(1)$, then 
\begin{align*}
n\mu(\cA_k)p S^{t(k-1)}(1/2+o(1)) \leq n\mu(\cA_k)p|\Gamma'_{k-1}(u)|(1/2 + o(1)) \\
 =n\mu(\cA_k)\left(1+o(1)\right)\left( 1-(1-\K_\ell(\cA_{k-1},\cA_k)p)^{|\Gamma'_{k-1}(u)|}\right)/2 \leq \E\; |\Gamma_k(u)|.
 \end{align*}
Hence by Lemma~\ref{lem:step} we have that
\[
\Pr{\cF_k^c | \cF_{k-1}} \leq e^{-\Omega(|\Gamma'_{k-1}(u)|)} \leq e^{- \Omega(S^{t(k-1)})}.
\]
Otherwise, if $| \Gamma_{k-1}(u)|p=\Omega(1)$, then  
\begin{align*}
\Theta(n \mu(\cA_k))&=n\mu(\cA_k)\left( 1-(1-\K_\ell(\cA_{k-1},\cA_k)p)^{|\Gamma_{k-1}(u)|}\right)/2 \leq \E\; |\Gamma_k(u)|,
\end{align*} 
and in this case we have
\[
\Pr{\cF_k^c | \cF_{k-1}} \leq e^{-\Omega(n)}.
\]
Putting together these bounds for different values of $k$ we obtain 
\begin{align*}
\Pr{\cF_\ell^c} 
&\leq \sum_{k=2}^\ell \Pr{\cF_k^c | \cF_{k-1}} \\
&\leq \sum_{k=2}^{\Phi -2} e^{- \Omega(S^{t(k-1)})} + \ell e^{-\Omega(n)} \\
&= e^{-\Omega(np)}=n^{-\omega}.
\end{align*}
The lemma follows.
\end{proof}

\section{Bounding the diameter}\label{sec:diameter}
We dedicate this section to the proof of Theorem \ref{thm:main}. We continue to write $G(n,\K, p)$ for an inhomogeneous random graph induced by an irreducible kernel $\K$  that is continuous $(\mu\otimes\mu)$-almost everywhere.  
We break up the proof of the theorem into six lemmas where we study the behavior of the diameter of  $G(n,\K, p)$  depending on the relationships between $\Phi, \Delta_u$ and $\Delta_\ell$.

\begin{lemma}\label{lem:DeltaPlusLB}
With high probability, $\diam G(n,\K,p) \geq \Delta_u$.
\end{lemma}
\begin{proof}
Consider a partition $\bA$ attaining $\Delta_u$. With probability $1$, there is no edge between any pair of vertices $w,x$ and any two elements $\cA_k, \cA_{\ell} \in \bA$, such that $w \in V(\cA_k), x \in V(\cA_{\ell})$, and $\cA_k,\cA_{\ell} \notin E(P_u(\bA))$. Note that with probability $1-e^{-\Omega(n)}$ we can find two vertices $u,v \in V(\cA_i), V(\cA_j)$ such that $\cA_i$ and $\cA_j$ are at distance $\Delta_u$ in $P_u(\bA)$, and hence the lemma follows.
\end{proof}

\begin{lemma}\label{lem:DeltaMinusUB}
Suppose $\Delta_\ell > \Phi$. With high probability, $\diam G(n,\K,p) \leq \Delta_\ell$.
\end{lemma}
\begin{proof}
We may assume $\Delta_\ell < \infty$, as otherwise there is nothing to prove.
Consider a partition $\bA$ attaining $\Delta_\ell$, and consider two arbitrary vertices $u,v \in V(\cA_i), V(\cA_j)$. Suppose first that in $P_\ell(\bA)$ there exists $\cA_r \in \bA$ such that $\cA_r,\cA_s,\cA_j$ is a walk (possibly reusing partitions and making zigzags) of length $2$ in $P_\ell(\bA)$, and  such that there exists a walk (possibly reusing partitions and making zigzags) of length $\Phi-2$ from $\cA_i$ to $\cA_r$. Then note that by Lemma~\ref{lem:path}, with probability at least $1-o(n^{-\omega})$
\begin{equation}\label{eq:Gamma1}
|\Gamma^{\Phi-2}(u) \cap V(\cA_r) | \geq (\norm{\bA}_{\mu}\K_\ell np/4)^{\Phi-2}.
\end{equation}
Also, with probability at least $1-o(n^{-\omega})$,
$|\Gamma(v) \cap V(\cA_s) | \geq \norm{\bA}_{\mu}\K_\ell np/4$. Assuming these conditions deterministically, the probability that there is no edge between $\Gamma^{\Phi-1}(u) \cap V(\cA_r)$ and $\Gamma(v) \cap V(\cA_s)$ is at most
\begin{equation}\label{eq:Gamma3}
(1-\K p)^{\norm{\bA}_{\mu}\K_\ell np/4)^{\Phi-1}} = e^{-\Omega(n^{\Phi-1}p^{\Phi})}=n^{-\omega},
\end{equation}
and hence with probability $1-n^{-\omega}$ we found a path of length $\Phi < \Delta_\ell$ between $u$ and $v$. Similarly, if there exists $\cA_r \in \bA$, where $\cA_r$ is such that $\cA_r,\cA_s,\cA_j$ is a walk of length $2$, and a walk of length $\Phi -1$ from $\cA_i$ to $\cA_r$, by Lemma~\ref{lem:path},
\[
|\Gamma^{\Phi-1}(u) \cap V(\cA_r) | \geq (\norm{\bA}_{\mu}\K_\ell np/4)^{\Phi-2},
\]
and by the same argument as before, we found a path of length $\Phi+1 \leq \Delta_\ell$ between $u$ and $v$. Otherwise, there must exist $\cA_r \in \bA$, such that $\cA_r,\cA_s,\cA_j$ is a path of length $2$, and such that the shortest path between $\cA_i$ and any vertex $\cA_r$ in $P_\ell(\bA)$ has length $\Phi \leq \ell \leq \Delta_\ell-2$. (If that were not the case, then $\cA_i$ and $\cA_j$ would be at distance bigger than $\Delta_\ell$, contradicting our assumption on $P_\ell(\bA)$. ) 
Again by Lemma~\ref{lem:path},
\[
 |\Gamma^{\ell}(u) \cap V(\cA_r) | \geq (\norm{\bA}_{\mu}\K n(p-p^2)/4)^{\Phi-2},
\]
 and by the same argument as before, with probability $1-o(n^{-\omega})$ we get a path of length $\ell+2 \leq \Delta_\ell$ between $u$ and $v$. By a union bound over all pairs of vertices, the statement follows.
\end{proof}

\pagebreak

\begin{lemma}\label{lem:PhiLB}
With high probability, $\diam G(n,\K,p) \geq \Phi$.
\end{lemma}
\begin{proof}
Note that for $p=\Theta(1)$ and $p < 1$, $\Phi=2$, $\Delta _u\geq 2$ and since $\diam G(n,\K,p) \geq \Delta_u$, we have $\diam G(n,\K,p) \geq \Phi$ in this case. Otherwise, for $p=o(1)$, using the same $X_1,\dots,X_n$, since $\K \leq 1$ we can couple $G(n,\K,p)$ so that it is a subgraph of $G(n,1,p)$. This can be done, for instance, by using uniform random variables $U_{ij} \in [0,1]$ and letting each edge be present if $U_{ij} < \K(X_i,X_j)p$ or $U_{ij} < p$. Note that $G(n,1,p)$ is nothing but $G(n,p)$. Recall that by Lemma \ref{lem:diameter} $\diam G(n,p)=\Phi$. Hence, with high probability we have $\diam G(n,\K,p) \geq \diam G(n,p)=\Phi$. 
\end{proof}

\begin{lemma}\label{lem:PhiPlusOneLB}
Suppose $\Phi \geq \Delta_u$.  If for every partition $\bA$ there exists  $i \leq j$, such that there exists no path of length exactly $\Phi$ between $\cA_i$ and $\cA_j$ in $P_u(\bA)$, then with high probability $\diam G(n,\K,p)  \geq \Phi+1$.
\end{lemma}
\begin{proof}
Fix an arbitrary partition $\bA$ and let $i \leq j$ as in the statement of the lemma. We will show that with high probability there exist two vertices $u \in V(\cA_i)$, $v \in V(\cA_j)$ such that $v$ cannot be reached by $u$ in exactly $\Phi$ steps. Note first that with probability $1$ there is no edge between two vertices $w \in V(\cA_k)$, $x \in V(\cA_{\ell})$ with $\cA_k,\cA_{\ell} \notin E(P_u(\bA))$, so we may assume this. Note that by Lemma \ref{lem:diameter}, if $(np)^{\Phi-1} - 2n \log n \to -\infty$, then there exists with high probability a pair of vertices $u,v$ such that $d_G(u,v) \geq \Phi$. Recall that we assume that the stronger condition~\eqref{eq:diambounds} holds. In this case, observe that with high probability there exist $\omega$ disjoint pairs of vertices such that they are all at distance exactly $\Phi$: to show this we use the idea of sprinkling new edges, introduced in \cite{AKS82}, combined with the coupling using the same $X_1,\dots,X_n$ described in the previous lemma. Indeed, if there were only constantly many, by adding fresh random edges distributed as $G(n,\K,C'p)$ for some large enough constant $C'=C'(p)$, we would on the one hand have a graph that can be coupled to be a subgraph of $G(n,\K,(C'+1)p)$, while still $((C'+1)pn)^{\Phi-1} - 2n \log n \to -\infty$. On the other hand, the number of new edges added is at least as big as the ones present before in $G(n,\K,p)$, since the number of new edges added is distributed as in $G(n,\K,C'p(1-p))$, and  $C'p(1-p) > p$ for any $p < 1$ and $C'$ sufficiently large. Hence, among the graph corresponding to the new edges there would also be at most constantly many pairs of vertices at distance $\Phi$, and the probability that  after adding the new edges, one vertex pair originally at distance $\Phi$ would remain at distance $\Phi$, tends to $0$, contradicting Lemma~\ref{lem:diameter}. Hence, in $G(n,\K,p)$ with high probability there are $\omega$ disjoint pairs of vertices at distance $\Phi$, and hence, with high probability there will be at least one pair of vertices $u,v$ satisfying $u \in V(\cA_i)$ and $v \in V(\cA_j)$. Since they cannot be at distance $\Phi$, they have to be at distance at least $\Phi+1$, yielding the lemma.
\end{proof}

\begin{lemma}\label{lem:PhiPlusOneUB}
Suppose $\Phi \geq \Delta_\ell$. With high probability $\diam G(n,\K,p) \leq \Phi+1$.
\end{lemma}
\begin{proof}
As in Lemma~\ref{lem:DeltaMinusUB}, consider a partition $\bA$ attaining $\Delta_\ell$, and consider two arbitrary vertices $u,v \in V(\cA_i), V(\cA_j)$. Observe that in $P_\ell(\bA)$ there exists $\cA_r \in \bA$ such that $\cA_r,\cA_s,\cA_j$ is a walk (possibly reusing partitions and making zigzags) of length $2$ in $P_\ell(\bA)$, and such that there exists a walk (possibly reusing partitions and making zigzags) of length $\Phi-2$ or $\Phi-1$ from $\cA_i$ to $\cA_r$. The proof is then as is in Lemma~\ref{lem:DeltaMinusUB}.
\end{proof}

\begin{lemma}\label{lem:PhiUB}
Suppose $\Phi \geq \Delta_\ell$. If there exists a partition $\bA$ such that for any $i \leq j$, there exists a path of length exactly $\Phi$ between $\cA_i$ and $\cA_j$ in $P_\ell(\bA)$, then with high probability $\diam G(n,\K,p) \leq \Phi$.
\end{lemma}
\begin{proof}
Consider such a partition $\bA$ and consider two vertices $u,v \in V(\cA_i), V(\cA_j)$:  by assumption, in $P_\ell(\bA)$ there exists $\cA_r \in \bA$ such that $\cA_r,\cA_s,\cA_j$ is a walk (possibly reusing partitions and making zigzags) of length $2$ in $P_\ell(\bA)$, and such that there exists a walk (possibly reusing partitions and making zigzags) of length exactly $\Phi-2$ from $\cA_i$ to $\cA_r$. The proof is then as is in Lemma~\ref{lem:DeltaMinusUB}.
\end{proof}

Finally, combining all the lemmas above we obtain the proof of the three statements in Theorem \ref{thm:main}.

\begin{proof}[\bf Proof of Theorem \ref{thm:main}]
 By combining Lemma~\ref{lem:DeltaPlusLB} and~\ref{lem:DeltaMinusUB} (i) follows. The first part of (ii) follows by Lemma~\ref{lem:PhiLB} and Lemma~\ref{lem:DeltaMinusUB}, and the second part follows by adding Lemma~\ref{lem:PhiPlusOneLB}. Finally, the first part of (iii) follows by Lemma~\ref{lem:PhiLB} and~\ref{lem:PhiPlusOneUB}; for the second part Lemma~\ref{lem:PhiPlusOneLB} and Lemma~\ref{lem:PhiUB} is used.
\end{proof}

\subsection*{Acknowledgements}
The work of the first author was made possible through generous support by NSERC. Part of this work was done at the Centre de Recerca Matem\`atica (CRM) in Barcelona. The authors are very grateful for its hospitality. The second author would also like to thank the Center for Mathematical Modeling (CMM, CNRS UMI 2807) at Universidad de Chile. This work was finished during his visit.


\bibliographystyle{amsplain}
\bibliography{diameter}

\end{document}